\documentclass[11pt]{article}

\usepackage{theorem}
\usepackage{amsmath,amssymb,amsfonts}
\usepackage{latexsym}
\usepackage{mathrsfs}
\usepackage{lscape}
\usepackage{enumerate}
\usepackage{hyperref}

\usepackage{longtable}

  \newtheorem{theorem}{Theorem}[section]

  \newtheorem{lemma}[theorem]{Lemma}
  
  \newtheorem{definition}[theorem]{Definition}
  \newtheorem{example}[theorem]{Example}
  \newtheorem{remark}[theorem]{Remark}
  \newtheorem{conjeture}[theorem]{Conjeture}
  \numberwithin{equation}{section}

\newenvironment{proof}{\noindent  Proof:\ }{\hspace*{\fill} $\Box $\\}

\def\N{\mathbb{N}}
\def\F{\mathbb{F}}
\def\P{\mathbb{P}}

\def\codeC{\mathscr{C}}

\title{Some constructions of cyclic and quasi cyclic subspaces codes}

\author{Ismael Gutierrez Garc\'ia$^{1}$ and Ivan Molina Naizir$^{2}$\\ 
\small{Department of Mathematics and Statistics}\\
\small{Universidad del Norte - Barranquilla, Colombia}\\
\small{isgutier@uninorte.edu.co$^{1}$}\\
\small{inaizir@uninorte.edu.co$^{2}$}}

\begin{document}
\maketitle

\begin{abstract}
In this paper we construct, using GAP System for Computational Discrete Algebra, some cyclic subspace codes, specially an optimal code over the finite field $\F_{2^{{10}}}$.  Further we introduce the $q$-analogous of a $m$-quasi cyclic subspace code over finite fields.
\end{abstract}


\section{Introduction}
Delivery of data in a conventional communication network like the Internet is performed by routing, i.e. intermediate nodes store and forward incoming information. Network coding is a scheme in which intermediate nodes are allowed to mix the incoming information before forwarding which results in a higher information flow rate and a high degree of robustness.

Network coding is a new research area in information theory that may have interesting applications in practical networking systems, like peer-to-peer content distribution network, bidirectional traffic in a wireless network, residential wireless mesh networks, Ad-hoc sensor networks, and others \cite{Fragouli}. As in classical algebraic coding theory, one of the most important research area in random network coding is the existence and construction of subspace codes with good parameters.

In this paper we consider the construction of a special class of subspace codes, namely cyclic codes. This codes were introduced by A. Kohnert and S. Kurz in \cite{Kurz} from an algebraic and geometric point of view. Later T. Etzion and A. Vardy in \cite{Etzion2} have defined them using the concept of cyclic shift for a subspace in a fixed Grassmannian over a finite field. 

Let $\F_q^n$ be the $n$-dimensional vector space over  the finite field, with $q$ elements, $\F_q$ (where $q$ is a prime power). We denote with $\P_q(n)$ the projective space of order $n$, that is, the set of all subspaces of $\F_q^n$, including the null space and $\F_q^n$ itself.

Let $\F_{q^n}$ be the extension field of $\F_q$ (of degree $n$). It is well known that we may regard $\F_{q^n}$ as a vector space of dimension $n$ over $\F_q$. That is, for a fixed basis, we can identifier every element of $\F_{q^n}$ with a $n$-tuple of elements in $\F_q$. Therefore we will not distinguish between $\F_{q^n}$ and $\F_q^n$.

For a fixed natural number $k\leq n$ we denote with $G_q(n,k)$ the set of all subspaces of $\F_q^n$ of dimension $k$ and we call it the $k$-Grassmannian. Then
\[\P_q(n) = \bigcup_{k=0}^n G_q(n,k).\]
The cardinality of the set $G_q(n,k)$ is given by the $q$-ary Gaussian coefficient ${n\brack k}_q$. It is well known that
\[{n\brack k}_q = \frac{(q^n-1)(q^{n-1}-1)\cdots (q^{n-k+1}-1)}{(q^k-1)(q^{k-1}-1)\cdots (q-1)}.\]

R. K\"otter and F. R. Kschischang \cite{Koetter1} have proven that the set $\P_q(n)$ endowed with the distance $d$ defined by
\begin{align*}
d(U, V) &= \dim(U + V) - \dim(U\cap V)\\
&= \dim(U) + \dim(V) - 2 \dim(U\cap V)
\end{align*}
is a metric space. The distance $d$ is called the subspace distance. A \texttt{subspace code} $\codeC$ is a non empty subset of $\P_q(n)$. The elements of $\codeC$ are called \texttt{codewords}. The \texttt{minimum distance} $d(\codeC)$ of a subspace code $\codeC$ is defined as usually, that is, as the smallest distance between any two different elements of $\codeC$.
Let $\codeC$ be a subspace code of minimum distance $d$. Then we say that $\codeC$ is a $[n,|\codeC|,d]$-code over $\F_q$ and $[n,|\codeC|,d]$ are its parameters. 

Constant dimension codes in network coding are the analogues of constant weight codes in classical coding theory. A \texttt{constant dimension code} $\codeC$ is just a non empty subset of $G_q(n,k)$.  If $\codeC \subseteq G_q(n,k)$ have minimum distance $d$, then we say that $\codeC$ is a $[n,k,|\codeC|,d]$-code over $\F_q$ and its parameters are given by $[n,k,|\codeC|,d]$. In this case, if $U, V\in \codeC$, then 
\[d(U,V) = 2k-2\dim(U\cap V).\] 
Thus $d(\codeC)$ is always an even number.

Recently T. Etzion et al. have presented in \cite{Etzion1} a method for constructing cyclic subspace codes, which includes some special kind of linearized polynomials, namely subspaces polynomials and also Frobenius mappings. 

Using groups actions, J. Rosenthal et al. \cite{Rosenthal} and H. Gluesing et al. \cite{Gluesing} studied a general version of cyclic subspace codes, the cyclic orbits codes. Specifically, they have used an action of the general linear group $\mathrm{GL}(n,q)$ over the set of all $k$-dimensional subspaces of $\F_q^n$ to define them. Nevertheless, in \cite{Rosenthal} no nontrivial construction of such codes is given and in \cite{Gluesing} is presented the construction of a cyclic subspace code, with no full length orbit. In both articles, it becomes relevant the following conjecture:

\begin{conjeture}\label{conjeture1}
For every  positive  integers $n, k$ such that $k< n/2$, there  exists  a subspace cyclic  code  of size $\frac{q^n-1}{q-1}$ in $G_q(n,k)$ and minimum distance $2k-2$.
\end{conjeture}

Let $\mathcal{A}_q(n,d)$, respectively $\mathcal{A}_q(n,d,k)$, denote the maximum number of codewords in an $[n,|\codeC|,d]$-code in $\P_q(n)$, respectively $[n,k,|\codeC|,d]$-code in $G_q(n,k)$. Due the minimum distance for a constant dimension code is always an even number, it suffices to consider $\mathcal{A}_q(n,d,k)$ for $d=2\delta$. T. Etzion and A. Vardy established in \cite{Etzion2} the following bound:
\begin{equation}\label{bound1}
\mathcal{A}_q(n,2 \delta+2,k)\leq  \frac{{n \brack k}_q}{{n-k+\delta \brack \delta}_q}.
\end{equation}

In this paper we present the construction of cyclic subspace codes. To do that, we have design and implemented in Java and C++ a set of algorithm. Furthermore, we use the library for finite fields contained in GAP System for Computational Discrete Algebra \cite{GAP}. Also we present the definition of $m$-quasi cyclic subspace codes as a natural generalization of cyclic codes and we show some results about the size of a quasi orbit.

\section{Cyclic subspace codes}

Let $\gamma$ be a primitive element of $\F_{q^n}$. A subspace code $\codeC\subseteq \P_q(n)$ is called \texttt{cyclic}, if it has the following property: 
\[\{0,\gamma^{i_1},\gamma^{i_2},\ldots, \gamma^{i_s}\}\in \codeC \Rightarrow \{0,\gamma^{i_1+1},\gamma^{i_2+1}, \ldots, \gamma^{i_s+1}\}\in \codeC.\]
(Assuming that $s=q^k-1$, with $k$ the dimension of the codeword).

Given a cyclic subspace code $\codeC$ and a $V\in \codeC$, we associate the corresponding binary characteristic vector $x_V = (x_0, x_1, \ldots, x_{q^n-2})$ as follows: 
\[x_j = \begin{cases}
1 & \text{if} \ \gamma^j \in V\\
0 & \text{if} \ \alpha^j \notin V.
\end{cases}\]

Then the set of all such characteristic vectors is closed under cyclic shifts. Note that the property of being cyclic does not depend on the choice of a primitive element in $\F_q$. 

A trivial construction of a cyclic subspace code is the following.

\begin{example}\label{ejemplo1}
Let $\gamma$ be a primitive root of $x^{10}+x^6+x^5+x^3+x^2+x+1$ and use this polynomial to generate the field $\F_{2^{10}}$. Let $\codeC\subseteq G_2(10,5)$ defined as follow:
\[\codeC := \{\alpha\F_{2^5}\mid \alpha\in \F_{2^{10}}^\ast\}.\]
That is, $\codeC$ is an orbit of $\F_{2^5}$ in $\F_{2^{10}}$. Note that the nonzero elements of $\F_{2^5}$ is a cyclic group generated by $\gamma^{33}$. Then the size of $\codeC$ is $\frac{2^{10}-1}{2^5-1} = 33$. On the other hand, note that if $U, V\in \codeC$, with $U\neq V$, then $dim(U\cap V) = 0$. Therefore, 
for $\alpha, \alpha'\in \F_{2^{10}}^\ast$ we have
\[d(\alpha\F_2^5,\alpha'\F_2^5) = 5+5-2\cdot 0=10.\]
Then $\codeC$ is an $[10,5,33,10]$-code. Specifically $\codeC$ consists of all
cyclic shifts of
\begin{align*}
& \{\gamma^0, \gamma^{33}, \gamma^{66}, \gamma^{99}, \gamma^{132}, \gamma^{165}, \gamma^{198}, \gamma^{231}, \gamma^{264}, \gamma^{297}, \gamma^{330}, \gamma^{363}, \gamma^{396}, \gamma^{429}, \\ 
& \ \  \gamma^{462}, \gamma^{495}, \gamma^{528}, \gamma^{561}, \gamma^{594}, \gamma^{627}, \gamma^{660}, \gamma^{693}, \gamma^{726}, \gamma^{759}, \gamma^{792}, \gamma^{825}, \gamma^{858}, \\ 
& \ \ \gamma^{891}, \gamma^{924}, \gamma^{957}, \gamma^{990}\}. 
\end{align*}
\end{example}

Note that in this example the zero vector was omitted from the set. Hereinafter, this will be explicitly deleted when we specifies the elements of a codeword of a cyclic or quasi-cyclic subspace code.

\begin{remark}
A subset $S$ of $G_q(n,k)$ is called a \texttt{spread} of $\F_q^n$, if holds
\begin{enumerate}
	\item[$(a)$] $V\cap W = \{0\}$, for all $V, W \in S$, and
	\item[$(b)$] for any $0\neq v\in \F_q^n$, exists an unique $V\in S$, such that $v\in V$.
\end{enumerate}
It is known, that spreads exists if and only if $k\mid n$. The code of Example \ref{ejemplo1} is a Spread of $\F_2^{10}$.
\end{remark}

\begin{example}\label{ejemplo2}
Let $\gamma$ be a primitive root of $x^{10}+x^6+x^5+x^3+x^2+x+1$ and use this polynomial to generate the field $\F_{2^{10}}$. Let $\codeC\subseteq G_2(10,3)$, which consists of all cyclic shifts of
\begin{align*}
& \{\gamma^0,\gamma^{37},\gamma^{104},\gamma^{170},\gamma^{251},\gamma^{269},\gamma^{576}\}\\
& \{\gamma^0,\gamma^{68},\gamma^{240},\gamma^{257},\gamma^{389},\gamma^{560},\gamma^{587}\}\\
& \{\gamma^0,\gamma^{126},\gamma^{169},\gamma^{243},\gamma^{452},\gamma^{487},\gamma^{707}\}\\
& \{\gamma^0,\gamma^{164},\gamma^{324},\gamma^{491},\gamma^{684},\gamma^{710},\gamma^{762}\}\\
& \{\gamma^0,\gamma^{59},\gamma^{295},\gamma^{418},\gamma^{537},\gamma^{631},\gamma^{718}\}\\
& \{\gamma^0,\gamma^{21},\gamma^{36},\gamma^{335},\gamma^{365},\gamma^{650},\gamma^{711}\}\\
& \{\gamma^0,\gamma^{70},\gamma^{173},\gamma^{380},\gamma^{457},\gamma^{654},\gamma^{811}\}\\
& \{\gamma^0,\gamma^{10},\gamma^{216},\gamma^{469},\gamma^{544},\gamma^{613},\gamma^{635}\}\\
& \{\gamma^0,\gamma^{105},\gamma^{161},\gamma^{289},\gamma^{424},\gamma^{517},\gamma^{565}\}\\
& \{\gamma^0,\gamma^{156},\gamma^{306},\gamma^{382},\gamma^{488},\gamma^{678},\gamma^{789}\}\\
& \{\gamma^0,\gamma^{31},\gamma^{42},\gamma^{131},\gamma^{143},\gamma^{241},\gamma^{399}\}\\
& \{\gamma^0,\gamma^{109},\gamma^{247},\gamma^{249},\gamma^{374},\gamma^{432},\gamma^{476}\}\\
& \{\gamma^0,\gamma^{124},\gamma^{246},\gamma^{524},\gamma^{527},\gamma^{577},\gamma^{672}\}\\
& \{\gamma^0,\gamma^{49},\gamma^{163},\gamma^{235},\gamma^{440},\gamma^{628},\gamma^{802}\}\\
& \{\gamma^0,\gamma^{60},\gamma^{176}, \gamma^{291},\gamma^{353},\gamma^{553},\gamma^{786}\}\\
& \{\gamma^0,\gamma^{107},\gamma^{286},\gamma^{441},\gamma^{482},\gamma^{578},\gamma^{793}\}\\
& \{\gamma^0,\gamma^{4},\gamma^{208},\gamma^{222},\gamma^{254},\gamma^{279},\gamma^{502}\}\\
& \{\gamma^0,\gamma^{298},\gamma^{515},\gamma^{523},\gamma^{543},\gamma^{588},\gamma^{607}\}\\
& \{\gamma^0,\gamma^{1},\gamma^{154},\gamma^{192},\gamma^{575},\gamma^{609},\gamma^{622}\}\\
& \{\gamma^0,\gamma^{7},\gamma^{23},\gamma^{175},\gamma^{434},\gamma^{497},\gamma^{580}\}\\
& \{\gamma^0,\gamma^{252},\gamma^{258},\gamma^{338}, \gamma^{518},\gamma^{680},\gamma^{719}\}.
\end{align*}

Therefore $\codeC$ is a $[10,3,21483,4]$-code and  we have 
\[21483 \leq \mathcal{A}_2(10,4,3) \leq 24893.\]

Nowadays this code is in second place in Tables of subspace codes, \cite{SubspaceCodesTable}, \cite{SubspaceCodesTable2}, organized by S. Kurz et al. at Mathematical Institute in University of Bayreuth - Germany.
\end{example}

\begin{example}\label{example3}
Let $\gamma$ be a primitive root of $x^8+x^4+x^3+x^2+1$ and use this polynomial to generate the field $\F_{2^{8}}$. Let $\codeC \subseteq G_2(8,4)$, which consists of all cyclic shifts of

\begin{align*}
& \{ \gamma^{0}, \gamma^{2}, \gamma^{29}, \gamma^{39}, \gamma^{49}, \gamma^{50}, \gamma^{60}, \gamma^{71}, \gamma^{74}, \gamma^{103}, \gamma^{106}, \gamma^{109}, \gamma^{132}, \gamma^{181}, \gamma^{197}\}\\
& \{ \gamma^{0}, \gamma^{2}, \gamma^{31}, \gamma^{45}, \gamma^{50}, \gamma^{91}, \gamma^{110}, \gamma^{123}, \gamma^{126}, \gamma^{163}, \gamma^{182}, \gamma^{183}, \gamma^{205}, \gamma^{207}, \gamma^{209}\}\\
& \{ \gamma^{0}, \gamma^{23}, \gamma^{64}, \gamma^{70}, \gamma^{79}, \gamma^{97}, \gamma^{110}, \gamma^{124}, \gamma^{126}, \gamma^{154}, \gamma^{174}, \gamma^{180}, \gamma^{190}, \gamma^{196}, \gamma^{201}\}\\
& \{ \gamma^{0}, \gamma^{1}, \gamma^{25}, \gamma^{38}, \gamma^{81}, \gamma^{94}, \gamma^{124}, \gamma^{155}, \gamma^{156}, \gamma^{159}, \gamma^{160}, \gamma^{169}, \gamma^{180}, \gamma^{184}, \gamma^{202}\}\\
& \{ \gamma^{0}, \gamma^{1}, \gamma^{25}, \gamma^{56}, \gamma^{64}, \gamma^{65}, \gamma^{70}, \gamma^{71}, \gamma^{89}, \gamma^{95}, \gamma^{109}, \gamma^{131}, \gamma^{162}, \gamma^{176}, \gamma^{203}\}\\
& \{ \gamma^{0}, \gamma^{16}, \gamma^{31}, \gamma^{45}, \gamma^{49}, \gamma^{88}, \gamma^{114}, \gamma^{145}, \gamma^{155}, \gamma^{159}, \gamma^{166}, \gamma^{171}, \gamma^{175}, \gamma^{197}, \gamma^{211}\}\\
& \{ \gamma^{0}, \gamma^{7}, \gamma^{30}, \gamma^{46}, \gamma^{66}, \gamma^{76}, \gamma^{87}, \gamma^{88}, \gamma^{89}, \gamma^{112}, \gamma^{113}, \gamma^{137}, \gamma^{167}, \gamma^{175}, \gamma^{203}\}\\
& \{ \gamma^{0}, \gamma^{5}, \gamma^{10}, \gamma^{21}, \gamma^{37}, \gamma^{40}, \gamma^{76}, \gamma^{84}, \gamma^{113}, \gamma^{114}, \gamma^{138}, \gamma^{143}, \gamma^{150}, \gamma^{166}, \gamma^{179}\}\\
& \{ \gamma^{0}, \gamma^{8}, \gamma^{16}, \gamma^{54}, \gamma^{69}, \gamma^{87}, \gamma^{125}, \gamma^{130}, \gamma^{145}, \gamma^{163}, \gamma^{167}, \gamma^{182}, \gamma^{194}, \gamma^{200}, \gamma^{208}\}\\
& \{ \gamma^{0}, \gamma^{40}, \gamma^{41}, \gamma^{53}, \gamma^{65}, \gamma^{80}, \gamma^{84}, \gamma^{98}, \gamma^{124}, \gamma^{139}, \gamma^{147}, \gamma^{157}, \gamma^{162}, \gamma^{168}, \gamma^{180}\}\\
& \{ \gamma^{0}, \gamma^{27}, \gamma^{59}, \gamma^{62}, \gamma^{82}, \gamma^{89}, \gamma^{90}, \gamma^{104}, \gamma^{114}, \gamma^{117}, \gamma^{122}, \gamma^{125}, \gamma^{166}, \gamma^{194}, \gamma^{203}\}\\
& \{ \gamma^{0}, \gamma^{19}, \gamma^{47}, \gamma^{62}, \gamma^{78}, \gamma^{80}, \gamma^{90}, \gamma^{92}, \gamma^{101}, \gamma^{128}, \gamma^{140}, \gamma^{168}, \gamma^{205}, \gamma^{207}, \gamma^{212}\}\\
& \{ \gamma^{0}, \gamma^{9}, \gamma^{28}, \gamma^{38}, \gamma^{47}, \gamma^{49}, \gamma^{93}, \gamma^{97}, \gamma^{101}, \gamma^{120}, \gamma^{158}, \gamma^{184}, \gamma^{190}, \gamma^{193}, \gamma^{197}\}\\
& \{ \gamma^{0}, \gamma^{7}, \gamma^{9}, \gamma^{57}, \gamma^{62}, \gamma^{64}, \gamma^{70}, \gamma^{72}, \gamma^{83}, \gamma^{90}, \gamma^{112}, \gamma^{120}, \gamma^{156}, \gamma^{169}, \gamma^{195}\}\\
& \{ \gamma^{0}, \gamma^{7}, \gamma^{47}, \gamma^{59}, \gamma^{79}, \gamma^{82}, \gamma^{91}, \gamma^{94}, \gamma^{101}, \gamma^{112}, \gamma^{148}, \gamma^{174}, \gamma^{202}, \gamma^{206}, \gamma^{209}\}\\
& \{ \gamma^{0}, \gamma^{6}, \gamma^{12}, \gamma^{49}, \gamma^{53}, \gamma^{58}, \gamma^{107}, \gamma^{127}, \gamma^{147}, \gamma^{149}, \gamma^{156}, \gamma^{169}, \gamma^{188}, \gamma^{191}, \gamma^{197}\}\\
& \{ \gamma^{0}, \gamma^{1}, \gamma^{8}, \gamma^{20}, \gamma^{25}, \gamma^{42}, \gamma^{76}, \gamma^{93}, \gamma^{113}, \gamma^{135}, \gamma^{144}, \gamma^{151}, \gamma^{158}, \gamma^{178}, \gamma^{200}\}\\
& \{ \gamma^{0}, \gamma^{13}, \gamma^{14}, \gamma^{38}, \gamma^{54}, \gamma^{85}, \gamma^{98}, \gamma^{99}, \gamma^{123}, \gamma^{139}, \gamma^{170}, \gamma^{183}, \gamma^{184}, \gamma^{208}, \gamma^{224}\}\\
& \{ \gamma^{0}, \gamma^{13}, \gamma^{14}, \gamma^{38}, \gamma^{54}, \gamma^{85}, \gamma^{98}, \gamma^{99}, \gamma^{123}, \gamma^{139}, \gamma^{170}, \gamma^{183}, \gamma^{184}, \gamma^{208}, \gamma^{224}\}\\
& \{ \gamma^{0}, \gamma^{9}, \gamma^{32}, \gamma^{35}, \gamma^{37}, \gamma^{85}, \gamma^{94}, \gamma^{117}, \gamma^{120}, \gamma^{122}, \gamma^{170}, \gamma^{179}, \gamma^{202}, \gamma^{205}, \gamma^{207}\}
\end{align*}

The code $\codeC$ is a $[8,4,4590,4]$-code and it holds
\[4590 \leq \mathcal{A}_2(8,4,4) \leq 6477.\]
This code is too in second place in Tables of subspace codes, \cite{SubspaceCodesTable}, \cite{SubspaceCodesTable2}.
\end{example}

\begin{definition}
Let $\alpha\in \F_{q^n}^\ast$ and $V\in G_q(n,k)$. The \texttt{cyclic shift} or the \texttt{orbit} of $V$ is defined as follows:
\[\alpha V := \{\alpha v\mid v\in V\}.\] 
Then a subspace code $\codeC\subseteq G_q(n,k)$ is called \texttt{cyclic}, if for all $0\neq \alpha\in \F_{q^n}$ and all subspace $V\in \codeC$ we have $\alpha V\in \codeC$.
\end{definition}

\begin{lemma}\cite[Lemma 9]{Etzion1}
If $V\in G_q(n,k)$, then 
\[|\{\alpha V\mid \alpha\in \F_{q^n}^\ast\}| = \frac{q^n-1}{q^t-1},\] 
for some natural number $t$, which divides $n$.
\end{lemma}

As an immediate consequence of the previous Lemma we have that the maximum size of an orbit is reached when $t = 1$. This justifies the following definition:

\begin{definition}
We say that $V\in G_q(n,k)$ has a \texttt{full length orbit} if
\[|\{\alpha V\mid \alpha\in \F_{q^n}^\ast\}| = \tfrac{q^n-1}{q-1}.\]
If $V$ does not have a full length orbit, then we say that it has a \texttt{degenerate orbit}. 
\end{definition}

It is clear that the set $\alpha V$ is again a subspace with the same dimension as $V$. If for $0\neq \alpha, \beta\in \F_{q^n}$ holds that $\alpha V\neq \beta V$, then we say that these cyclic shifts are \texttt{distinct}.

\begin{remark}
In the code of example \ref{example3} the first 17 orbits are full length orbits and the remaining 3 are orbits of 85 subspaces.
\end{remark}

\section{Classification of some Orbits in $\F_q^n$}

The following tables have been calculated using GAP System for Computational Discrete Algebra and some algorithms programmed in Java and C++ languages.

\begin{enumerate}
\item[$(\mathrm{I})$] $n=6$

\begin{enumerate}[$(1)$]
	\item Number of orbits.
	\begin{center}
		\begin{tabular}{|c|c|c|c|}
			\hline 
			$k\setminus d$ & 2 & 4 & 6\\ 
			\hline 
			1 & 1 & 0 & 0 \\ 
			\hline 
			2 & 10 & 1 & 0\\ 
			\hline 
			3 & 14 & 8 &  1\\ 
			\hline 
		\end{tabular}
	\end{center}
	\vskip0.5cm
	
	\item Number of full length orbits.
	\begin{center}
		\begin{tabular}{|c|c|c|c|}
			\hline 
			$k\setminus d$ & 2 & 4 & 6\\ 
			\hline 
			1 & 1 & 0 & 0 \\ 
			\hline 
			2 & 10 & 0 & 0\\ 
			\hline 
			3 & 14 & 8 &  0\\ 
			\hline 
		\end{tabular}
	\end{center}
	\vskip0.5cm
	
	\item Number of degenerate orbits.
	\begin{enumerate}[$(a)$]
		\item Number of orbits with 21 subspaces
	\begin{center}
		\begin{tabular}{|c|c|c|c|}
			\hline 
			$k\setminus d$ & 2 & 4 & 6\\ 
			\hline 
			1 & 0 & 0 & 0 \\ 
			\hline 
			2 & 0 & 1 & 0\\ 
			\hline 
			3 & 0 & 0 &  0\\ 
			\hline 
		\end{tabular}
	\end{center}
	\vskip0.5cm
	
	\item Number of orbits with 9 subspaces: There exist just one, the Spread code $ \{\alpha\F_{2^3}\mid \alpha\in \F_{2^6}^\ast\}$.
\end{enumerate}
\end{enumerate}

\item[$(\mathrm{II})$] $n=7$

\begin{enumerate}[$(1)$]
	\item Number of orbits.
	\begin{center}
		\begin{tabular}{|c|c|c|}
			\hline 
			$k\setminus d$ & 2 & 4 \\ 
			\hline 
			1 & 1 & 0 \\ 
			\hline 
			2 & 21 & 0 \\ 
			\hline 
			3 & 21 & 72 \\ 
			\hline 
		\end{tabular}
	\end{center}
\end{enumerate}

\item[$(\mathrm{III})$] $n=8$

\begin{enumerate}[$(1)$]
\item Number of orbits.
\begin{center}
\begin{tabular}{|c|c|c|c|c|c|}
\hline 
$k\setminus d$ & 2 & 4 & 6 & 8  \\ 
\hline 
1 & 1 & 0 & 0 & 0  \\ 
\hline 
2 & 42 & 1 & 0 & 0  \\ 
\hline 
3 & 61 & 320 & 0 & 0  \\ 
\hline 
4 & 40 & 750 & 0 & 1  \\ 
\hline 
\end{tabular}
\end{center}
\vskip0.5cm

\item Number of full length orbits.
\begin{center}
\begin{tabular}{|c|c|c|c|c|c|}
\hline 
$k\setminus d$ & 2 & 4 & 6 & 8  \\ 
\hline 
1 & 1 & 0 & 0 & 0  \\ 
\hline 
2 & 42 & 0 & 0 & 0  \\ 
\hline 
3 & 61 & 320 & 0 & 0  \\ 
\hline 
4 & 40 & 746 & 0 & 0  \\ 
\hline 
\end{tabular}
\end{center}
\vskip0.5cm

\item Number of degenerate orbits.
\begin{enumerate}[$(a)$]
	\item Number of orbits with 85 subspaces
	
	\begin{center}
		\begin{tabular}{|c|c|c|c|c|c|}
			\hline 
			$k\setminus d$ & 2 & 4 & 6 & 8  \\ 
			\hline 
			1 & 0 & 0 & 0 & 0  \\ 
			\hline 
			2 & 0 & 1 & 0 & 0  \\ 
			\hline 
			3 & 0 & 0 & 0 & 0  \\ 
			\hline 
			4 & 0 & 4 & 0 & 0  \\ 
			\hline 
		\end{tabular}
	\end{center}
	\vskip0.5cm
	\item Number of orbits with 17 subspaces: There exist just one, the Spread code $ \{\alpha\F_{2^4}\mid \alpha\in \F_{2^8}^\ast\}$.
\end{enumerate}
\end{enumerate}

\item[$(\mathrm{IV})$] $n=9$
\begin{enumerate}[$(1)$]
	\item Number of orbits
	\begin{center}
		\begin{tabular}{|c|c|c|c|}
			\hline 
			$k\setminus d$ & 2 & 4 & 6\\ 
			\hline 
			1 & 1 & 0 & 0 \\ 
			\hline 
			2 & 85 & 0 & 0\\ 
			\hline 
			3 & 84 & 1458 & 1\\ 
			\hline 
			4 & 93 & 5736 & 648 \\ 
			\hline 
		\end{tabular}
	\end{center}
	\vskip0.5cm
	
	\item Number of full length orbits
	\begin{center}
		\begin{tabular}{|c|c|c|c|}
			\hline 
			$k\setminus d$ & 2 & 4 & 6\\ 
			\hline 
			1 & 1 & 0 & 0 \\ 
			\hline 
			2 & 85 & 0 & 0\\ 
			\hline 
			3 & 84 & 1458 & 0\\ 
			\hline 
			4 & 93 & 5736 & 648 \\ 
			\hline 
		\end{tabular}
	\end{center}
	\vskip0.5cm
	
	\item Number of degenerate orbits.
	\begin{enumerate}[$(a)$]
	\item Number of orbits with 73 subspaces
	\begin{center}
		\begin{tabular}{|c|c|c|c|}
			\hline 
			$k\setminus d$ & 2 & 4 & 6\\ 
			\hline 
			1 & 0 & 0 & 0 \\ 
			\hline 
			2 & 0 & 0 & 0\\ 
			\hline 
			3 & 0 & 0 & 1\\ 
			\hline 
			4 & 0 & 0 & 0 \\ 
			\hline 
		\end{tabular}
	\end{center}
\end{enumerate}
\end{enumerate}

\item[$(\mathrm{V})$] $n=10$
\begin{enumerate}[$(1)$]
\item Number of orbits

\begin{center}
\begin{tabular}{|c|c|c|c|c|c|}
\hline 
$k\setminus d$ & 2 & 4 & 6 & 8 & 10 \\ 
\hline 
1 & 1 & 0 & 0 & 0 & 0 \\ 
\hline 
2 & 170 & 1 & 0 & 0 & 0 \\ 
\hline 
3 & 255 & 5950 & 0 & 0 & 0 \\ 
\hline 
4 & 166 & 31487 & 20894 & 0 & 0 \\ 
\hline 
5 & 522 & 41772 & 64472 & 0 & 1 \\ 
\hline 
\end{tabular}
\end{center}
\vskip0.5cm

\item Number of full length orbits

\begin{center}
\begin{tabular}{|c|c|c|c|c|c|}
\hline 
$k\setminus d$ & 2 & 4 & 6 & 8 & 10 \\ 
\hline 
1 & 1 & 0 & 0 & 0 & 0 \\ 
\hline 
2 & 170 & 0 & 0 & 0 & 0 \\ 
\hline 
3 & 255 & 5950 & 0 & 0 & 0 \\ 
\hline 
4 & 166 & 31470 & 20894 & 0 & 0 \\ 
\hline 
5 & 522 & 41772 & 64472 & 0 & 0 \\ 
\hline 
\end{tabular}
\end{center}
\vskip0.5cm

\item Number of degenerate orbits.
\begin{enumerate}[$(a)$]
	\item Number of orbits with 341 subspaces
	
	\begin{center}
		\begin{tabular}{|c|c|c|c|c|c|}
			\hline 
			$k\setminus d$ & 2 & 4 & 6 & 8 & 10 \\ 
			\hline 
			1 & 0 & 0 & 0 & 0 & 0 \\ 
			\hline 
			2 & 0 & 1 & 0 & 0 & 0 \\ 
			\hline 
			3 & 0 & 0 & 0 & 0 & 0 \\ 
			\hline 
			4 & 0 & 17 & 0 & 0 & 0 \\ 
			\hline 
			5 & 0 & 0 & 0 & 0 & 0 \\ 
			\hline 
		\end{tabular}
	\end{center}
	\vskip0.5cm
	\item Number of orbits with 33 subspaces: There exist just one, the Spread code $ \{\alpha\F_{2^5}\mid \alpha\in \F_{2^{10}}^\ast\}$.
\end{enumerate}
\end{enumerate}
\end{enumerate}

\begin{remark}
The conjecture 1 can be refuted by the result in second table given that does not exist a full length orbit in $\F_2^{10}$ with dimension $5$ and minimum distance $8$. (A counterexample to the conjecture was initially given by H. Gluesing et al. in \cite{Gluesing}).
\end{remark}

\section{Duality}

We consider now the usual inner product $(\cdot,\cdot)$ defined on $\F_q^n$, that is, for $x=(x_1,\ldots,x_n)$, $y=(y_1,\ldots, y_n)$ in $\F_q^n$
\[(x,y) = \sum_{j=0}^n x_jy_j.\]
If $U\in \P_q(n)$ and  $\dim U=k$, then the \texttt{orthogonal complement} of $U$ is the $(n-k)$-dimensional subspace of $\F_q^n$ defined as follows:
\[U^\perp = \{x\in \F_q^n\mid  (u,x) = 0, \forall u\in U\}.\]

\begin{definition}
Let $\codeC$ be a subspaces code. The \texttt{dual code} of $\codeC$, noted by $\codeC^\perp$, is defined by
\[\codeC^\perp  = \{V^\perp \mid V\in \codeC\} \subseteq \P_q(n).\]
\end{definition}
 
It is known that the distance between the subspaces $U$ and $V$ is reflected to the distance between the orthogonal subspaces $U^\perp$ and $V^\perp$.

\begin{lemma}[Duality]\cite[Lemma 13]{Etzion2}\label{dualidad}
	\label{dualidad2}
Let $\codeC$ be a subspaces code of type $[n,k,|\codeC|,d]$. Then the dual code $\codeC^\perp$ is a code of type $[n,n-k,|\codeC|,d]$.
\end{lemma}

\begin{remark}
In classical coding theory is verified that, if $C$ is a $[n,k]$ cyclic code over $\F_q$, then so is its dual. As we can see below, this  is not true in context of subspaces codes. Let $\gamma$ be a primitive root of $x^5+x^2+1$ and use this polynomial to generate the field $\F_{2^5}$. The following table present a binary cyclic subspace code $\codeC$ with parameters $[5,2,31,2]$ and its corresponding dual code $\codeC^\perp$, which is not cyclic.
\end{remark}

\begin{center}
	\begin{longtable}{l||l|}
		\ \ \ \ \ \ \ \ $\codeC$ & \ \ \ \ \ \ \ \ \ \ \ \ \ \ \ $\codeC^\perp$\\
		\hline
		&\\
		$\{ \gamma^{0}, \gamma^{13}, \gamma^{14}\}$ & $\{\gamma^{1}, \gamma^{7}, \gamma^{9}, \gamma^{12}, \gamma^{20}, \gamma^{21}, \gamma^{28}\}$\\
		$\{ \gamma^{1}, \gamma^{14}, \gamma^{15}\}$ & $\{ \gamma^{5}, \gamma^{7}, \gamma^{10}, \gamma^{14}, \gamma^{20}, \gamma^{21}, \gamma^{29}\}$\\
		$\{ \gamma^{2}, \gamma^{15}, \gamma^{16}\}$ & $\{ \gamma^{6}, \gamma^{10}, \gamma^{16}, \gamma^{18}, \gamma^{21}, \gamma^{29}, \gamma^{30}\}$\\
		$\{ \gamma^{3}, \gamma^{16}, \gamma^{17}\}$ & $\{ \gamma^{2}, \gamma^{10}, \gamma^{11}, \gamma^{18}, \gamma^{22}, \gamma^{28}, \gamma^{30}\}$\\
		$\{ \gamma^{4}, \gamma^{17}, \gamma^{18}\}$ & $\{\gamma^{2}, \gamma^{3}, \gamma^{11}, \gamma^{18}, \gamma^{20}, \gamma^{23}, \gamma^{27}\}$\\
		$\{ \gamma^{5}, \gamma^{18}, \gamma^{19}\}$ & $\{\gamma^{3}, \gamma^{4}, \gamma^{11}, \gamma^{15}, \gamma^{21}, \gamma^{23}, \gamma^{26}\}$\\
		$\{ \gamma^{6}, \gamma^{19}, \gamma^{20}\}$ & $\{\gamma^{0}, \gamma^{4}, \gamma^{10}, \gamma^{12}, \gamma^{15}, \gamma^{23}, \gamma^{24}\}$\\
		$\{ \gamma^{7}, \gamma^{20}, \gamma^{21}\}$ & $\{\gamma^{0}, \gamma^{1}, \gamma^{13}, \gamma^{14}, \gamma^{15}, \gamma^{18}, \gamma^{24}\}$\\
		$\{ \gamma^{8}, \gamma^{21}, \gamma^{22}\}$ & $\{ \gamma^{1}, \gamma^{5}, \gamma^{11}, \gamma^{13}, \gamma^{16}, \gamma^{24}, \gamma^{25}\}$\\
		$\{ \gamma^{9}, \gamma^{22}, \gamma^{23}\}$ & $\{ \gamma^{5}, \gamma^{6}, \gamma^{13}, \gamma^{17}, \gamma^{23}, \gamma^{25}, \gamma^{28}\}$\\
		$\{ \gamma^{10}, \gamma^{23}, \gamma^{24}\}$ & $\{ \gamma^{6}, \gamma^{15}, \gamma^{17}, \gamma^{19}, \gamma^{20}, \gamma^{22}, \gamma^{25}\}$\\
		$\{ \gamma^{11}, \gamma^{24}, \gamma^{25}\}$ & $\{\gamma^{8}, \gamma^{17}, \gamma^{19}, \gamma^{21}, \gamma^{22}, \gamma^{24}, \gamma^{27}\}$\\
		$\{ \gamma^{12}, \gamma^{25}, \gamma^{26}\}$ & $\{ \gamma^{8}, \gamma^{9}, \gamma^{10}, \gamma^{13}, \gamma^{19}, \gamma^{26}, \gamma^{27}\}$\\
		$\{ \gamma^{13}, \gamma^{26}, \gamma^{27}\}$ & $\{\gamma^{7}, \gamma^{8}, \gamma^{9}, \gamma^{12}, \gamma^{18}, \gamma^{25}, \gamma^{26}\}$\\
		$\{ \gamma^{14}, \gamma^{27}, \gamma^{28}\}$ & $\{\gamma^{7}, \gamma^{9}, \gamma^{11}, \gamma^{12}, \gamma^{14}, \gamma^{17}, \gamma^{29}\}$\\
		$\{ \gamma^{15}, \gamma^{28}, \gamma^{29}\}$ & $\{\gamma^{7}, \gamma^{14}, \gamma^{16}, \gamma^{19}, \gamma^{23}, \gamma^{29}, \gamma^{30}\}$\\
		$\{ \gamma^{16}, \gamma^{29}, \gamma^{30}\}$ & $\{\gamma^{2}, \gamma^{8}, \gamma^{15}, \gamma^{16}, \gamma^{28}, \gamma^{29}, \gamma^{30}\}$\\
		$\{ \gamma^{0}, \gamma^{17}, \gamma^{30}\}$ & $\{\gamma^{2}, \gamma^{3}, \gamma^{9}, \gamma^{20}, \gamma^{24}, \gamma^{28}, \gamma^{30}\}$\\
		$\{ \gamma^{0}, \gamma^{1}, \gamma^{18}\}$ & $\{\gamma^{2}, \gamma^{3}, \gamma^{4}, \gamma^{7}, \gamma^{13}, \gamma^{20}, \gamma^{21}\}$\\
		$\{ \gamma^{1}, \gamma^{2}, \gamma^{19}\}$ & $\{\gamma^{0}, \gamma^{3}, \gamma^{4}, \gamma^{10}, \gamma^{21}, \gamma^{25}, \gamma^{29}\}$\\
		$\{ \gamma^{2}, \gamma^{3}, \gamma^{20}\}$ & $\{\gamma^{0}, \gamma^{1}, \gamma^{4}, \gamma^{10}, \gamma^{17}, \gamma^{18}, \gamma^{30}\}$\\
		$\{ \gamma^{3}, \gamma^{4}, \gamma^{21}\}$ & $\{\gamma^{0}, \gamma^{1}, \gamma^{2}, \gamma^{5}, \gamma^{11}, \gamma^{18}, \gamma^{19}\}$\\
		$\{ \gamma^{4}, \gamma^{5}, \gamma^{22}\}$ & $\{ \gamma^{1}, \gamma^{3}, \gamma^{5}, \gamma^{6}, \gamma^{8}, \gamma^{11}, \gamma^{23}\}$\\
		$\{ \gamma^{5}, \gamma^{6}, \gamma^{23}\}$ & $\{\gamma^{4}, \gamma^{5}, \gamma^{6}, \gamma^{9}, \gamma^{15}, \gamma^{22}, \gamma^{23}\}$\\
		$\{ \gamma^{6}, \gamma^{7}, \gamma^{24}\}$ & $\{\gamma^{0}, \gamma^{6}, \gamma^{7}, \gamma^{15}, \gamma^{22}, \gamma^{24}, \gamma^{27}\}$\\
		$\{ \gamma^{7}, \gamma^{8}, \gamma^{25}\}$ & $\{\gamma^{1}, \gamma^{13}, \gamma^{22}, \gamma^{24}, \gamma^{26}, \gamma^{27}, \gamma^{29}\}$\\
		$\{ \gamma^{8}, \gamma^{9}, \gamma^{26}\}$ & $\{ \gamma^{5}, \gamma^{12}, \gamma^{13}, \gamma^{25}, \gamma^{26}, \gamma^{27}, \gamma^{30}\}$\\
		$\{ \gamma^{9}, \gamma^{10}, \gamma^{27}\}$ & $\{\gamma^{2}, \gamma^{6}, \gamma^{12}, \gamma^{14}, \gamma^{17}, \gamma^{25}, \gamma^{26}\}$\\
		$\{ \gamma^{10}, \gamma^{11}, \gamma^{28}\}$ & $\{\gamma^{3}, \gamma^{12}, \gamma^{14}, \gamma^{16}, \gamma^{17}, \gamma^{19}, \gamma^{22}\}$\\ 
		$\{ \gamma^{11}, \gamma^{12}, \gamma^{29}\}$ & $\{\gamma^{4}, \gamma^{8}, \gamma^{14}, \gamma^{16}, \gamma^{19}, \gamma^{27}, \gamma^{28}\}$\\
		$\{ \gamma^{12}, \gamma^{13}, \gamma^{30}\}$& $\{\gamma^{0}, \gamma^{8}, \gamma^{9}, \gamma^{16}, \gamma^{20}, \gamma^{26}, \gamma^{28}\}$\\
	\end{longtable}
\end{center}

\begin{example}
Let $\gamma$ be a primitive root of $x^6+x^4+x^3+x+1$ and use this polynomial to generate the field $\F_{2^6}$. The following table present the spread Code $\codeC = \{\alpha\F_{2^3}\mid \alpha\in \F_{2^6}^\ast\}$ and its corresponding dual code $\codeC^\perp$.

\begin{longtable}{l||l}
\hline
\ \ \ \ \ \ \ \ \ \ \ \ \ \ \ $\codeC$ & \ \ \ \ \ \ \ \ \ \ \ \ \ \ \ $\codeC^\perp$ \\ 
\hline
$\{\gamma^{0}, \gamma^{9}, \gamma^{18}, \gamma^{27}, \gamma^{36}, \gamma^{45}, \gamma^{54} \}$ & $\{\gamma^{3}, \gamma^{17}, \gamma^{19}, \gamma^{22}, \gamma^{32}, \gamma^{47}, \gamma^{51}\}$ \\
$\{\gamma^{1}, \gamma^{10}, \gamma^{19}, \gamma^{28}, \gamma^{37}, \gamma^{46}, \gamma^{55}\}$	& $\{\gamma^{12}, \gamma^{23}, \gamma^{27}, \gamma^{34}, \gamma^{35}, \gamma^{46}, \gamma^{58}\}$\\
$\{\gamma^{2}, \gamma^{11}, \gamma^{20}, \gamma^{29}, \gamma^{38}, \gamma^{47}, \gamma^{56}\}$ & $\{\gamma^{16}, \gamma^{45}, \gamma^{52}, \gamma^{53}, \gamma^{59}, \gamma^{60}, \gamma^{61}\}$\\
$\{\gamma^{3}, \gamma^{12}, \gamma^{21}, \gamma^{30}, \gamma^{39}, \gamma^{48}, \gamma^{57}\}$ & $\{\gamma^{4}, \gamma^{8}, \gamma^{10}, \gamma^{30}, \gamma^{39}, \gamma^{54}, \gamma^{57}\}$\\
$\{\gamma^{4}, \gamma^{13}, \gamma^{22}, \gamma^{31}, \gamma^{40}, \gamma^{49}, \gamma^{58}\}$ & $\{\gamma^{1}, \gamma^{5}, \gamma^{25}, \gamma^{36}, \gamma^{42}, \gamma^{48}, \gamma^{62}\}$\\
$\{\gamma^{5}, \gamma^{14}, \gamma^{23}, \gamma^{32}, \gamma^{41}, \gamma^{50}, \gamma^{59}\}$ & $\{\gamma^{0}, \gamma^{2}, \gamma^{6}, \gamma^{26}, \gamma^{37}, \gamma^{43}, \gamma^{49}\}$\\
$\{\gamma^{6}, \gamma^{15}, \gamma^{24}, \gamma^{33}, \gamma^{42}, \gamma^{51}, \gamma^{60}\}$ & $\{\gamma^{7}, \gamma^{9}, \gamma^{13}, \gamma^{33}, \gamma^{44}, \gamma^{50}, \gamma^{56}\}$\\
$\{\gamma^{7}, \gamma^{16}, \gamma^{25}, \gamma^{34}, \gamma^{43}, \gamma^{52}, \gamma^{61}\}$ & $\{\gamma^{11}, \gamma^{14}, \gamma^{20}, \gamma^{24}, \gamma^{38}, \gamma^{40}, \gamma^{55}\}$\\
$\{\gamma^{8}, \gamma^{17}, \gamma^{26}, \gamma^{35}, \gamma^{44}, \gamma^{53}, \gamma^{62}\}$ & $\{\gamma^{15}, \gamma^{18}, \gamma^{21}, \gamma^{28}, \gamma^{29}, \gamma^{31}, \gamma^{41}\}$\\
\end{longtable}

Once more, it can be seen that the dual code $\codeC^\perp$ is non-cyclic.
\end{example}

\section{Quasi-cyclic subspaces codes}

\begin{definition}
Let $\gamma$ be a primitive element of $\F_{q^n}$ and $m$ a natural number with $m\mid (q^n-1)$. 
A subspace code $\codeC\subseteq \P_q(n)$ is called $m$-\texttt{quasi cyclic}, if holds the following property: 
\[\{0,\gamma^{i_1},\gamma^{i_2},\ldots, \gamma^{i_s}\}\in \codeC \Rightarrow \{0,\gamma^{i_1+m},\gamma^{i_2+m}, \ldots, \gamma^{i_s+m}\}\in \codeC.\]
(Assuming that $s=q^k-1$, with $k$ the dimension of the codeword).
\end{definition}

Now we present a natural generalization for the definition of orbit of a subspace and for the length of an orbit.

\begin{definition}
Let $\alpha\in \F_{q^n}^\ast$, $m$ a natural number with $m\mid q^n-1$ and $V\in G_q(n,k)$. 
The \texttt{$m$-quasi cyclic shift} or the \texttt{$m$-quasi orbit} of a subspace $V$ is defined by
\[\alpha^m V := \{\alpha^m v\mid v\in V\}.\] 
Then a subspace code $\codeC\subseteq G_q(n,k)$ is called $m$-\texttt{quasi cyclic}, if for all $0\neq \alpha\in \F_{q^n}$ and all subspace 
$V\in \codeC$ we have $\alpha^m V\in \codeC$.
\end{definition}

It is clear that the set $\alpha^m V$ is again a subspace with the same dimension as $V$.  

The demonstrations idea of the following lemma is the same as the presented by T. Etzion at al. in \cite[Lemma 9]{Etzion1} for the case $m=1$. 
This is obtained only by performing basic modifications to the cited one.

\begin{lemma}\label{lema9}
If $m$ a natural number with $m\mid q^n-1$ and $V\in G_q(n,k)$, then 
\[|\{\alpha^m V\mid \alpha\in \F_{q^n}^\ast\}| = \frac{1}{m}\big(\frac{q^n-1}{q^t-1}\big), \]
for some natural number $t$, which divides $n$.
\end{lemma}

\begin{proof}
Let $\gamma$ be a primitive element in $\F_{q^n}$, that is, $\F_{q^n}^\ast =\langle \gamma\rangle$ and let $l$ the smallest natural number with $\gamma^{lm}V=V$.
It is clear that $lm\mid q^n-1$. Let now $0\leq s < l$ and $i\in \N$, then
\begin{align*}
\gamma^{iml+s}V & = \gamma^s(\gamma^{iml}V)\\
&= \gamma^s(\gamma^{ml}\cdots \gamma^{ml})V\\
&= \gamma^sV.
\end{align*}
That is, for each natural number $i$ and for each $0\leq s < l$ is verified that $\gamma^sV= \gamma^{iml+s}V$.
Additionally, for every $0\leq s_1, s_2< l$ the sets 
\[A_{s_j} := \{\gamma^{iml+s_j}\mid i\in \N\}\]
satisfy that $|A_{s_1}| = |A_{s_2}|$. In fact, given that $q^n-1 = wml$, for some $w\in \N$, then we have
\[A_{s_j} = \{\gamma^{s_j}, \gamma^{ml+s_j},\ldots, \gamma^{ml(w-1)+s_j}\}.\]
Therefore $|A_{s_1}| = |A_{s_2}|=w$. Let $\gamma^{iml}, \gamma^{rml}\in A_0$, for some $i, r\in \N$. Since $A_0 = \{\gamma^{iml}\mid i\in \N\}$, it follows that
\[(\gamma^{iml} + \gamma^{rml})V \subseteq \gamma^{iml}V + \gamma^{rml}V = V+V = V,\]
and therefore $\gamma^{iml} + \gamma^{rml}\in A_0$. It is clear that $A_0$ is closed under multiplication, then we have that $\langle \gamma^{ml}\rangle$ is the multiplicative
group of a subfield of $\F_{q^n}$, say $\F_{q^t}$, for some natural number $t$, which divides $n$. Then
\[|\{\alpha^m V\mid \alpha\in \F_{q^n}^\ast\}| = l= \tfrac{q^n-1}{mw} =\tfrac{1}{m}\big(\tfrac{q^n-1}{q^t-1}\big),\]
which proves affirmation.
\end{proof}

An immediate consequence of Lemma \ref{lema9} is that the largest possible size of an $m$-quasi orbit is $\frac{1}{m}\big(\frac{q^n-1}{q-1}\big)$. This justifies the following definition:

\begin{definition}
We say that the subspace $V\in G_q(n,k)$ has a \texttt{full length $m$-quasi orbit}, if
\[|\{\alpha^m V\mid \alpha\in \F_{q^n}^\ast\}| = \tfrac{1}{m}\big(\tfrac{q^n-1}{q-1}\big).\]
In other case we say that it has a \texttt{degenerate $m$-quasi orbit}. 
\end{definition}

\begin{example}
A 3-quasi cyclic orbit with parameters $[8,4,2992,4]$. 

Let $\gamma$ be a primitive root of $x^8+x^4+x^3+x^2+1$ and use this polynomial to generate the field $\F_{2^{8}}$. 
Let $\codeC$ be the code in $G_2(8,4)$ which consists of all 3-cyclic shifts of
\begin{longtable}{l}
$\{ \gamma^{0}, \gamma^{19}, \gamma^{58}, \gamma^{62}, \gamma^{90}, \gamma^{92}, \gamma^{93}, \gamma^{107}, \gamma^{117}, \gamma^{122}, \gamma^{125}, \gamma^{128}, \gamma^{140}, \gamma^{158}, \gamma^{194}\}$\\
$\{ \gamma^{0}, \gamma^{6}, \gamma^{13}, \gamma^{47}, \gamma^{99}, \gamma^{101}, \gamma^{118}, \gamma^{149}, \gamma^{156}, \gamma^{163}, \gamma^{164}, \gamma^{169}, \gamma^{182}, \gamma^{188}, \gamma^{191}\}$\\
$\{ \gamma^{1}, \gamma^{27}, \gamma^{42}, \gamma^{58}, \gamma^{59}, \gamma^{60}, \gamma^{62}, \gamma^{72}, \gamma^{83}, \gamma^{84}, \gamma^{108}, \gamma^{110}, \gamma^{158}, \gamma^{187}, \gamma^{199}\}$\\
$\{ \gamma^{1}, \gamma^{18}, \gamma^{20}, \gamma^{59}, \gamma^{68}, \gamma^{69}, \gamma^{80}, \gamma^{93}, \gamma^{108}, \gamma^{126}, \gamma^{152}, \gamma^{175}, \gamma^{179}, \gamma^{195}, \gamma^{217}\}$\\
$\{ \gamma^{2}, \gamma^{25}, \gamma^{62}, \gamma^{77}, \gamma^{89}, \gamma^{95}, \gamma^{96}, \gamma^{120}, \gamma^{134}, \gamma^{160}, \gamma^{166}, \gamma^{169}, \gamma^{198}, \gamma^{201}, \gamma^{204}\}$\\
$\{ \gamma^{2}, \gamma^{22}, \gamma^{44}, \gamma^{49}, \gamma^{83}, \gamma^{96}, \gamma^{103}, \gamma^{125}, \gamma^{126}, \gamma^{150}, \gamma^{162}, \gamma^{182}, \gamma^{185}, \gamma^{204}, \gamma^{208}\}$\\
$\{ \gamma^{2}, \gamma^{30}, \gamma^{39}, \gamma^{40}, \gamma^{51}, \gamma^{64}, \gamma^{92}, \gamma^{120}, \gamma^{150}, \gamma^{166}, \gamma^{181}, \gamma^{186}, \gamma^{195}, \gamma^{199}, \gamma^{208}\}$\\
$\{ \gamma^{2}, \gamma^{12}, \gamma^{22}, \gamma^{23}, \gamma^{33}, \gamma^{42}, \gamma^{44}, \gamma^{47}, \gamma^{64}, \gamma^{78}, \gamma^{86}, \gamma^{92}, \gamma^{115}, \gamma^{153}, \gamma^{180}\}$\\
$\{ \gamma^{0}, \gamma^{7}, \gamma^{38}, \gamma^{52}, \gamma^{72}, \gamma^{79}, \gamma^{94}, \gamma^{112}, \gamma^{141}, \gamma^{156}, \gamma^{169}, \gamma^{174}, \gamma^{184}, \gamma^{195}, \gamma^{202}\}$\\
$\{ \gamma^{2}, \gamma^{9}, \gamma^{46}, \gamma^{48}, \gamma^{63}, \gamma^{78}, \gamma^{81}, \gamma^{96}, \gamma^{114}, \gamma^{115}, \gamma^{139}, \gamma^{176}, \gamma^{188}, \gamma^{204}, \gamma^{217}\}$\\
$\{ \gamma^{0}, \gamma^{5}, \gamma^{37}, \gamma^{40}, \gamma^{67}, \gamma^{74}, \gamma^{84}, \gamma^{95}, \gamma^{103}, \gamma^{135}, \gamma^{138}, \gamma^{144}, \gamma^{176}, \gamma^{179}, \gamma^{216}\}$\\
$\{ \gamma^{1}, \gamma^{7}, \gamma^{14}, \gamma^{48}, \gamma^{100}, \gamma^{102}, \gamma^{119}, \gamma^{150}, \gamma^{157}, \gamma^{164}, \gamma^{165}, \gamma^{170}, \gamma^{183}, \gamma^{189}, \gamma^{192}\}$\\
$\{ \gamma^{1}, \gamma^{17}, \gamma^{33}, \gamma^{36}, \gamma^{88}, \gamma^{96}, \gamma^{98}, \gamma^{109}, \gamma^{126}, \gamma^{146}, \gamma^{162}, \gamma^{168}, \gamma^{177}, \gamma^{191}, \gamma^{195}\}$\\
$\{ \gamma^{1}, \gamma^{3}, \gamma^{51}, \gamma^{61}, \gamma^{63}, \gamma^{72}, \gamma^{91}, \gamma^{110}, \gamma^{111}, \gamma^{127}, \gamma^{133}, \gamma^{135}, \gamma^{164}, \gamma^{178}, \gamma^{183}\}$\\
$\{ \gamma^{1}, \gamma^{7}, \gamma^{24}, \gamma^{33}, \gamma^{36}, \gamma^{53}, \gamma^{75}, \gamma^{104}, \gamma^{142}, \gamma^{144}, \gamma^{151}, \gamma^{188}, \gamma^{192}, \gamma^{197}, \gamma^{205}\}$\\
$\{ \gamma^{0}, \gamma^{12}, \gamma^{27}, \gamma^{31}, \gamma^{45}, \gamma^{65}, \gamma^{87}, \gamma^{104}, \gamma^{127}, \gamma^{155}, \gamma^{159}, \gamma^{162}, \gamma^{167}, \gamma^{171}, \gamma^{211}\}$\\
$\{ \gamma^{2}, \gamma^{6}, \gamma^{36}, \gamma^{45}, \gamma^{55}, \gamma^{66}, \gamma^{72}, \gamma^{102}, \gamma^{112}, \gamma^{123}, \gamma^{128}, \gamma^{138}, \gamma^{149}, \gamma^{156}, \gamma^{203}\}$\\
$\{ \gamma^{0}, \gamma^{26}, \gamma^{41}, \gamma^{57}, \gamma^{58}, \gamma^{59}, \gamma^{61}, \gamma^{71}, \gamma^{82}, \gamma^{83}, \gamma^{107}, \gamma^{109}, \gamma^{157}, \gamma^{186}, \gamma^{198}\}$\\
$\{ \gamma^{0}, \gamma^{6}, \gamma^{23}, \gamma^{53}, \gamma^{55}, \gamma^{58}, \gamma^{63}, \gamma^{74}, \gamma^{89}, \gamma^{103}, \gamma^{107}, \gamma^{147}, \gamma^{191}, \gamma^{196}, \gamma^{203}\}$\\
$\{ \gamma^{2}, \gamma^{6}, \gamma^{8}, \gamma^{49}, \gamma^{56}, \gamma^{102}, \gamma^{103}, \gamma^{127}, \gamma^{157}, \gamma^{161}, \gamma^{165}, \gamma^{184}, \gamma^{193}, \gamma^{196}, \gamma^{210}\}$\\
$\{ \gamma^{1}, \gamma^{14}, \gamma^{16}, \gamma^{18}, \gamma^{34}, \gamma^{56}, \gamma^{64}, \gamma^{66}, \gamma^{69}, \gamma^{77}, \gamma^{100}, \gamma^{114}, \gamma^{155}, \gamma^{163}, \gamma^{202}\}$\\
$\{ \gamma^{2}, \gamma^{4}, \gamma^{9}, \gamma^{32}, \gamma^{52}, \gamma^{68}, \gamma^{74}, \gamma^{91}, \gamma^{114}, \gamma^{130}, \gamma^{142}, \gamma^{158}, \gamma^{171}, \gamma^{197}, \gamma^{205}\}$\\
$\{ \gamma^{2}, \gamma^{39}, \gamma^{43}, \gamma^{48}, \gamma^{55}, \gamma^{82}, \gamma^{95}, \gamma^{139}, \gamma^{149}, \gamma^{159}, \gamma^{160}, \gamma^{165}, \gamma^{170}, \gamma^{181}, \gamma^{184}\}$\\
$\{ \gamma^{2}, \gamma^{4}, \gamma^{42}, \gamma^{43}, \gamma^{52}, \gamma^{59}, \gamma^{63}, \gamma^{67}, \gamma^{85}, \gamma^{86}, \gamma^{110}, \gamma^{159}, \gamma^{163}, \gamma^{164}, \gamma^{188}\}$\\
$\{ \gamma^{1}, \gamma^{5}, \gamma^{6}, \gamma^{7}, \gamma^{30}, \gamma^{31}, \gamma^{55}, \gamma^{67}, \gamma^{95}, \gamma^{101}, \gamma^{139}, \gamma^{182}, \gamma^{192}, \gamma^{203}, \gamma^{209}\}$\\
$\{ \gamma^{2}, \gamma^{28}, \gamma^{32}, \gamma^{41}, \gamma^{68}, \gamma^{108}, \gamma^{112}, \gamma^{127}, \gamma^{128}, \gamma^{145}, \gamma^{150}, \gamma^{152}, \gamma^{196}, \gamma^{200}, \gamma^{208}\}$\\
$\{ \gamma^{1}, \gamma^{5}, \gamma^{11}, \gamma^{20}, \gamma^{22}, \gamma^{38}, \gamma^{70}, \gamma^{73}, \gamma^{93}, \gamma^{101}, \gamma^{115}, \gamma^{131}, \gamma^{167}, \gamma^{180}, \gamma^{196}\}$\\
$\{ \gamma^{1}, \gamma^{41}, \gamma^{60}, \gamma^{61}, \gamma^{73}, \gamma^{76}, \gamma^{83}, \gamma^{85}, \gamma^{94}, \gamma^{133}, \gamma^{159}, \gamma^{188}, \gamma^{196}, \gamma^{200}, \gamma^{205}\}$\\
$\{ \gamma^{2}, \gamma^{21}, \gamma^{60}, \gamma^{64}, \gamma^{92}, \gamma^{94}, \gamma^{95}, \gamma^{109}, \gamma^{119}, \gamma^{124}, \gamma^{127}, \gamma^{130}, \gamma^{142}, \gamma^{160}, \gamma^{196}\}$\\
$\{ \gamma^{1}, \gamma^{8}, \gamma^{39}, \gamma^{53}, \gamma^{73}, \gamma^{80}, \gamma^{95}, \gamma^{113}, \gamma^{142}, \gamma^{157}, \gamma^{170}, \gamma^{175}, \gamma^{185}, \gamma^{196}, \gamma^{203}\}$\\
$\{ \gamma^{1}, \gamma^{32}, \gamma^{46}, \gamma^{59}, \gamma^{89}, \gamma^{108}, \gamma^{115}, \gamma^{125}, \gamma^{136}, \gamma^{145}, \gamma^{167}, \gamma^{176}, \gamma^{181}, \gamma^{190}, \gamma^{220}\}$\\
$\{ \gamma^{2}, \gamma^{8}, \gamma^{12}, \gamma^{23}, \gamma^{41}, \gamma^{87}, \gamma^{93}, \gamma^{97}, \gamma^{108}, \gamma^{126}, \gamma^{172}, \gamma^{178}, \gamma^{182}, \gamma^{193}, \gamma^{211}\}$\\
$\{ \gamma^{2}, \gamma^{15}, \gamma^{16}, \gamma^{40}, \gamma^{56}, \gamma^{87}, \gamma^{100}, \gamma^{101}, \gamma^{125}, \gamma^{141}, \gamma^{172}, \gamma^{185}, \gamma^{186}, \gamma^{210}, \gamma^{226}\}$\\
$\{ \gamma^{0}, \gamma^{9}, \gamma^{32}, \gamma^{35}, \gamma^{37}, \gamma^{85}, \gamma^{94}, \gamma^{117}, \gamma^{120}, \gamma^{122}, \gamma^{170}, \gamma^{179}, \gamma^{202}, \gamma^{205}, \gamma^{207}\}$\\
$\{ \gamma^{0}, \gamma^{7}, \gamma^{19}, \gamma^{27}, \gamma^{49}, \gamma^{85}, \gamma^{92}, \gamma^{104}, \gamma^{112}, \gamma^{134}, \gamma^{170}, \gamma^{177}, \gamma^{189}, \gamma^{197}, \gamma^{219}\}$\\
$\{ \gamma^{0}, \gamma^{17}, \gamma^{34}, \gamma^{51}, \gamma^{68}, \gamma^{85}, \gamma^{102}, \gamma^{119}, \gamma^{136}, \gamma^{153}, \gamma^{170}, \gamma^{187}, \gamma^{204}, \gamma^{221}, \gamma^{238}\}$\\
\end{longtable}

The first 35 orbits are full length 3-quasi orbits (85 subspaces) and the remaining one is an orbit with 17 subspaces. Then the code $\codeC$ is a $[8,4,2992,4]$-code.

Note that, despite the factor 1/3 guaranteed by Lemma \ref{lema9}, this code has a large number of codewords. Not so far compared to the cyclic code of Example \ref{ejemplo2}.
\end{example}

\begin{example}[Quasi-cyclic self-dual subspace codes]
Let $\gamma$ be a primitive root of $x^4+x+1$ and use this polynomial to generate the field $\F_{2^4}$
\begin{enumerate}[$(a)$]
\item Let $\codeC$ be the code
\[\codeC = \{\{\gamma^{2}, \gamma^{3}, \gamma^{6}\}, \{ \gamma^{5}, \gamma^{6}, \gamma^{9}\}, \{ \gamma^{8}, \gamma^{9}, \gamma^{12}\}, \{ \gamma^{0}, \gamma^{11}, \gamma^{12}\}, \{ \gamma^{0}, \gamma^{3}, \gamma^{14}\}\}.\]
Than is $\codeC$ a 3-quasi cyclic subspace code with parameters $[4,2,5,2]$.
\item Let $\codeC$ be the code
\[\codeC = \{\{\gamma^{2}, \gamma^{7}, \gamma^{12}\}, \{\gamma^{4}, \gamma^{9}, \gamma^{14}\} \}.\]

$\codeC$ is a 5-quasi cyclic orbit code with parameters $[4,2,2,4]$. 

These codes are the unique quasi-cyclic self-dual codes in projective space $\P_2(4)$.
\end{enumerate}
\end{example}

\begin{example}[Quasi-cyclic self-dual subspace code]
The following code $\codeC$ is the unique quasi-cyclic self-dual code in projective space $\P_2(6)$. 
\begin{align*}
\codeC =  & \{\{ \gamma^{9}, \gamma^{24}, \gamma^{30}, \gamma^{33}, \gamma^{43}, \gamma^{50}, \gamma^{51}\}, \{ \gamma^{1}, \gamma^{8}, \gamma^{9}, \gamma^{30}, \gamma^{45}, \gamma^{51}, \gamma^{54}\}\\
& \ \ \{\gamma^{3}, \gamma^{9}, \gamma^{12}, \gamma^{22}, \gamma^{29}, \gamma^{30}, \gamma^{51}\}\}.
\end{align*}
$\codeC$  is a 21-quasi cyclic orbit code with parameters $[6,3,3,2]$. 
\end{example}

\begin{example}
Let $\gamma$ be a primitive root of $x^8+x^4+x^3+x^2+1$ and use this polynomial to generate the field $\F_{2^{8}}$.
Let $\codeC$ be the code
\begin{align*}
\codeC &= \{\{\gamma^{27}, \gamma^{34}, \gamma^{46}, \gamma^{54}, \gamma^{76}, \gamma^{112}, \gamma^{119}, \gamma^{131}, \gamma^{139}, \gamma^{161}, \gamma^{197}, \gamma^{204}, \gamma^{216}, \gamma^{224}, \gamma^{246}\} \\
& \ \ \ \ \ \ \{\gamma^{5}, \gamma^{27}, \gamma^{63}, \gamma^{70}, \gamma^{82}, \gamma^{90}, \gamma^{112}, \gamma^{148}, \gamma^{155}, \gamma^{167}, \gamma^{175}, \gamma^{197}, \gamma^{233}, \gamma^{240}, \gamma^{252}\}\}.
\end{align*}
$\codeC$  is a 85-quasi cyclic orbit code with parameters $[8,4,2,4]$.  This is the unique quasi-cyclic, self-dual code in projective space $\P_2(8)$.
\end{example}

\section{Classification of some Quasi-orbits  in $\F_q^n$}

The following tables have been calculated using GAP System for Computational Discrete Algebra and some algorithms programmed in Java language.

\begin{enumerate}
\item[$(\mathrm{I})$] $n=8$
\begin{enumerate}[$(1)$]
	\item $m=3$. 
	\begin{enumerate}[$(a)$]
		\item Number of Full length 3-Quasi Orbits with 85 subspaces. 
		
		\begin{center}
			\begin{tabular}{|c|c|c|}
				\hline 
				$k\setminus d$ & 2 & 4 \\ 
				\hline 
				1 & 3 & 0 \\ 
				\hline 
				2 & 102 & 24 \\ 
				\hline 
				3 & 99 & 1044 \\ 
				\hline 
				4 & 96 & 2262 \\ 
				\hline 
			\end{tabular}
		\end{center}
		\vskip0.5cm
		\item Number of degenerate 3-Quasi Orbits: 0. 
	\end{enumerate}

	\item $m=5$. 
	\begin{enumerate}[$(a)$]
		\item Number of 5-Quasi Orbits.
		
		\begin{center}
			\begin{tabular}{|c|c|c|c|}
				\hline 
				$k\setminus d$ & 2 & 4 & 6\\ 
				\hline 
				1 & 5 & 0 & 0\\ 
				\hline 
				2 & 120 & 95 & 0\\ 
				\hline 
				3 & 225 & 1680 & 0\\ 
				\hline 
				4 & 120 & 3590 & 240\\ 
				\hline 
			\end{tabular}
		\end{center}
		\vskip0.5cm
		\item Number of Full length 5-Quasi Orbits with 51 subspaces.
		
		\begin{center}
			\begin{tabular}{|c|c|c|c|}
				\hline 
				$k\setminus d$ & 2 & 4 & 6\\ 
				\hline 
				1 & 5 & 0 & 0 \\ 
				\hline 
				2 & 120 & 90 & 0\\ 
				\hline 
				3 & 225 & 1680 & 0 \\ 
				\hline 
				4 & 120 & 3570 & 240\\ 
				\hline 
			\end{tabular}
		\end{center}
		\vskip0.5cm
		\item Number of degenerate 5-Quasi Orbits with 17 subspaces.
		
		\begin{center}
			\begin{tabular}{|c|c|c|c|}
				\hline 
				$k\setminus d$ & 2 & 4 & 6 \\ 
				\hline 
				1 & 0 & 0 & 0\\ 
				\hline 
				2 & 0 & 5 & 0\\ 
				\hline 
				3 & 0 & 0 & 0\\ 
				\hline 
				4 & 0 & 20 & 0\\
				\hline 
			\end{tabular}
		\end{center}
	\end{enumerate}
	
	\item $m=15$. 
	\begin{enumerate}[$(a)$]
		\item Number of Full length 15-Quasi Orbits with 17 subspaces. 
		
	\begin{center}
		\begin{tabular}{|c|c|c|c|}
			\hline 
			$k\setminus d$ & 2 & 4 & 6\\ 
			\hline 
			1 & 15 & 0 & 0 \\ 
			\hline 
			2 & 120 & 510 & 0\\ 
			\hline 
			3 & 120 & 4380 & 1215 \\ 
			\hline 
			4 & 120 & 6000 & 5670 \\ 
			\hline 
		\end{tabular}
	\end{center}
	\vskip0.5cm
	\item Number of degenerate 15-Quasi Orbits: 0. 
\end{enumerate}

\item $m=17$. 
\begin{enumerate}[$(a)$]
	\item Number of 17-Quasi Orbits.
	
	\begin{center}
		\begin{tabular}{|c|c|c|c|c|c|}
			\hline 
			$k\setminus d$ & 0 & 2 & 4 & 6 & 8\\ 
			\hline 
			1 & 0 & 17 & 0 & 0 & 0\\ 
			\hline 
			2 & 0 & 34 & 697 & 0 & 0\\ 
			\hline 
			3 & 0 & 357 & 2040 & 4080 & 0 \\ 
			\hline 
			4 & 17 & 0 & 4930 & 8160 & 340 \\
			\hline 
		\end{tabular}
	\end{center}
	\vskip0.5cm
	\item Number of Full length 17-Quasi Orbits with 15 subspaces.
	
	\begin{center}
		\begin{tabular}{|c|c|c|c|c|}
			\hline 
			$k\setminus d$ & 2 & 4 & 6 & 8\\ 
			\hline 
			1 & 17 & 0 & 0 & 0\\ 
			\hline 
			2 & 34 & 680 & 0 & 0\\ 
			\hline 
			3 & 357 & 2040 & 4080 & 0 \\ 
			\hline 
			4 & 0 & 4930 & 8160 & 272\\ 
			\hline 
		\end{tabular}
	\end{center}
	\vskip0.5cm
	
	\item Number of degenerate 17-Quasi Orbits with 5 subspaces.
	
	\begin{center}
		\begin{tabular}{|c|c|c|c|c|}
			\hline 
			$k\setminus d$ & 2 & 4 & 6 & 8\\ 
			\hline 
			1 & 0 & 0 & 0 & 0\\ 
			\hline 
			2 & 0 & 17 & 0 & 0\\ 
			\hline 
			3 & 0 & 0 & 0 & 0 \\ 
			\hline 
			4 & 0 & 0 & 0 & 68 \\
			\hline 
		\end{tabular}
	\end{center}
	\vskip0.5cm
	
	\item Number of degenerate 17-Quasi Orbits with 1 subspaces.
	
	\begin{center}
		\begin{tabular}{|c|c|}
			\hline 
			$k\setminus d$ & 0 \\ 
			\hline 
			1 & 0\\ 
			\hline 
			2 & 0\\ 
			\hline 
			3 & 0\\ 
			\hline 
			4 & 17\\
			\hline 
		\end{tabular}
	\end{center}
\end{enumerate}
	
	\item $m=51$. 
	\begin{enumerate}[$(a)$]
		\item Number of Full length 51-Quasi Orbits with 5 subspaces. 
		
	\begin{center}
		\begin{tabular}{|c|c|c|c|c|}
			\hline 
			$k\setminus d$ & 2 & 4 & 6 & 8\\ 
			\hline 
			1 & 51 & 0 & 0 & 0 \\ 
			\hline 
			2 & 102 & 2040 & 0 & 0\\ 
			\hline 
			3 & 51 & 6120 & 13260 & 0 \\ 
			\hline 
			4 & 0 & 5610 & 32640 & 1836 \\ 
			\hline 
		\end{tabular}
	\end{center}
	\vskip0.5cm
	\item Number of degenerate 51-Quasi Orbits: 0. 
\end{enumerate}

	\item $m=85$. 
	\begin{enumerate}[$(a)$]	
	\item Number of 85-Quasi Orbits.
	\begin{center}
		\begin{tabular}{|c|c|c|c|c|c|}
			\hline 
			$k\setminus d$ & 0 & 2 & 4 & 6 & 8\\ 
			\hline 
			1 & 0 & 85 & 0 & 0 & 0 \\
			\hline 
			2 & 85 & 0 & 3570 & 0 & 0\\ 
			\hline 
			3 & 0 & 1785 & 0 & 30600 & 0\\ 
			\hline 
			4 & 340 & 0 & 17850 & 0 & 48960 \\
			\hline 
		\end{tabular}
	\end{center}	
	\vskip0.5cm
	
	\item Number of Full length 85-Quasi Orbits with 3 subspaces.
	
	\begin{center}
		\begin{tabular}{|c|c|c|c|c|}
			\hline 
			$k\setminus d$ & 2 & 4 & 6 & 8\\ 
			\hline 
			1 & 85 & 0 & 0 & 0\\ 
			\hline 
			2 & 0 & 3570 & 0 & 0\\ 
			\hline 
			3 & 1785 & 0 & 30600 & 0 \\ 
			\hline 
			4 & 0 & 17850 & 0 & 48960 \\
			\hline 
		\end{tabular}
	\end{center}	
	\vskip0.5cm
	
	\item Number of degenerate 85-Quasi Orbits with 1 subspaces.
	
	\begin{center}
		\begin{tabular}{|c|c|}
			\hline 
			$k\setminus d$ & 0\\ 
			\hline 
			1 & 0\\ 
			\hline 
			2 & 85\\ 
			\hline 
			3 & 0\\ 
			\hline 
			4 & 340\\
			\hline 
		\end{tabular}
	\end{center}
\end{enumerate}	
\end{enumerate}
\end{enumerate}

\section{Conclusions and future work}
In this paper we present some classifications of full length and degenerate orbits of a subspace in projective space $\P_q(n)$, for $n=!!!!$. In addition, $m$-quasi cyclic subspaces codes are defined, as a natural generalization of cyclic subspace codes and some classifications of $m$-quasi orbits are shown.

For future investigations, we can consider the generalization of well knows results about cyclic subspaces codes with degenerate orbits and the connection between $m$-quasi cyclic subspaces codes and orbits codes.

\section{Acknowledgements}
The first Author would like to thank the hospitality of the Institute for Algebra and Geometry at Otto von Guericke University - Magdeburg, where part of this work was carried out. Also acknowledges and thanks the financial support of the Deutscher Akademischer Austausch Dienst.

\bibliographystyle{abbrv} 
\bibliography{Isgutier.07.2016}

\begin{thebibliography}{10}

\bibitem{Etzion1}
E.~Ben-Sasson, T.~Etzion, A.~Gabizon, and N.~Raviv.
\newblock Subspace polynomials and cyclic subspace codes.
\newblock {\em arXiv:1404.7739}, 2015.

\bibitem{Etzion2}
T.~Etzion and A.~Vardy.
\newblock Error-correcting codes in projective space.
\newblock {\em IEEE Transactions on Information Theory}, 57(2), 2011.

\bibitem{Fragouli}
C.~Fragouli, J.-Y.~L. Boudec, and J.~Widmer.
\newblock Network coding: An instant primer.
\newblock {\em ACM SIGCOMM Computer Communication Review}, 36:63--68, 2006.

\bibitem{GAP}
GAP.
\newblock Groups, algorithms, programming - a system for computational discrete
  algebra.
\newblock \url{http://www.gap-system.org/}.

\bibitem{Gluesing}
H.~Gluesing-Luerssen, K.~Morrison, and C.~Troha.
\newblock Cyclic orbit codes and stabilizer subfields.
\newblock {\em Advances in Mathematics of Communications}, 9(2):177--197, May
  2015.

\bibitem{SubspaceCodesTable}
D.~Heinlein, M.~Kiermaier, S.~Kurz, and A.~Wassermann.
\newblock Tables of subspace codes.
\newblock \url{http://subspacecodes.uni-bayreuth.de/}.

\bibitem{SubspaceCodesTable2}
D.~Heinlein, M.~Kiermaier, S.~Kurz, and A.~Wassermann.
\newblock Tables of subspace codes.
\newblock {\em arXiv preprint arXiv:1601.02864}, 2016.

\bibitem{Koetter1}
R.~Koetter and F.~R. Kschischang.
\newblock Coding for errors and erasures in random network coding.
\newblock {\em IEEE Transactions on Information Theory}, 54:3579 -- 3591, 2008.

\bibitem{Kurz}
A.~Kohnert and S.~Kurz.
\newblock Construction of large constant dimension codes with a prescribed
  minimum distance.
\newblock {\em Mathematical Methods in Computer Science. Lecture Notes in
  Computer Science}, 5393:31--42, 2008.

\bibitem{Rosenthal}
A.-L. Trautmann, F.~Manganiello, M.~Braun, and J.~Rosenthal.
\newblock Cyclic orbit codes.
\newblock {\em IEEE Transactions on Information Theory}, 59:7386--7404, 2013.

\end{thebibliography}

\appendix
\section{Appendix}

\subsection{Used Methods for constructing subspace cyclic codes}

\subsubsection{Getting orbits}
GAP is used to calculate the subspaces of a specific Grassmannian $G_q(n,k)$ for some $n,q,k$.

For each space the following procedure is performed: Let $U\in G_q(n,k)$.
\begin{enumerate}
	\item Set $b = 0$
	\item For each $\alpha^i \in U$, $b = b + 2^i$
	\item Save $b_U:=b$
\end{enumerate}
This process makes easier the management of subspaces.

Then using Java Language, this information is collected and for each number $b_U$ perform the following process:
\begin{enumerate}
	\item Take an empty list $L_U$ and insert the element $b_U$
	\item Set $b:=b_U$, if $b_U < q^n-2$ then it is added $b_{\alpha U}:=2*b$ to the list, else $b_{\alpha U}:=2*(b-q^n-2)+1$ is added to the list.
	\item In $L$ are all elements of the orbit associated with $U$
	\item Save $L$
\end{enumerate}
This process is performed in parallel with the search of subspace in GAP, to optimize the process and stop GAP, once we have the number of orbits in $G_q(n, k)$.

\subsubsection{Calculating the minimum distance of an Orbit}
To calculate the distance between two subspaces the following process is performed.

Let's take $b_1, b_2$ two numbers representing two subspaces, the operation $\mathrm{AND}(\&)$ calculates the representative number of subspace intersection of the spaces associated with $b_1$ and $b_2$; The operation $\mathrm{AND}$ is the operation that takes a bit of each number and the corresponding bit set to 1 if both corresponding bits are 1 and 0 otherwise.

So, if $U \cap V = W$, then $b_U$ $\&$ $b_V = b_W$

For the dimension of a space, it takes 
\[\dim(U) = \log_q{\text{\# bits to 1 in its representation}}\]
for the minimum distance from an orbit, we took a $ U $ basis space and calculate the distance with all other spaces in orbit; Is not necessary to make all comparisons since:
\[d(\alpha^i U,\alpha^j U)  =  d(U,\alpha^{j-i} U).\]

\subsubsection{Getting subspace cyclic codes}
For cyclic subspaces codes we calculate first the minimum distance of joining any two orbits.

To find the minimum distance of joining 2 orbits, take any subspace $U$ in the first orbit, and calculate the distance between this space and all spaces in the second orbit, the minimum will be the minimum distance between the two orbits. It is not necessary to make all comparisons since:
\[d(\alpha^i U,\alpha^j V)  =  d(U,\alpha^{j-i} V).\]

After this we get a graph $G$ with vertices $V(G)$ the orbits of the Grassmannian, and for every pair of vertices, we add an edge if and only if the corresponding orbits have minimum distance (the distance of join both orbits) higher to a $d$ preestablished.

Taking this graph each clique (A clique is a complete graph, i.e., any pair of vertices are adjacent). Is a cyclic network code.
In this graph we can set some bounds on the maximum size of network code:

$\Omega(G)$ be noticed as the number of the clique (i.e. the size of the largest clique). Find cliques in a graph is considered an NP-complete problem
\end{document}